\documentclass[12pt]{amsart}
\usepackage{amsmath}
\usepackage{amssymb}

\usepackage{latexsym}

\usepackage{comment}
\usepackage{enumitem}
\usepackage{xcolor}

\DeclareMathOperator{\supp}{supp}
\DeclareMathOperator{\id}{id}
\DeclareMathOperator{\Fl}{Fl}

\newcommand{\CG}{{\mathcal{G}}}          

\newcommand{\CEF}{{\mathcal{E}}}            
\newcommand{\CGF}{{\mathcal{G}}}     
\newcommand{\OCA}[1]{{\mathcal A}_{#1}}       
\newcommand{\CNF}{{\mathcal{N}}}            
\newcommand\mol{\phi}
\newcommand{\cD}{\mathcal{D}}               
\newcommand{\eps}{\varepsilon}              

\newcommand{\abso}[1]{\left\lvert#1\right\rvert}

\newcommand{\Nat}{\mathbb{N}}                
\newcommand{\Lie}{\mathcal{L}}              

\def\CE{{\mathcal E}}               
\def\CMF{{\mathcal E}_M}            

\newcommand{\coleq}{\mathrel{\mathop:}=}

\begin{document}

\newcommand{\rd}{\mathrm{d}}
\newcommand{\dx}{\rd \x}
\newcommand{\dy}{\rd \y}

\newtheorem{theorem}{\bf Theorem}
\newtheorem{definition}[theorem]{\bf Definition}
\newtheorem{condition}[theorem]{\bf Condition}
\newtheorem{corollary}[theorem]{\bf Corollary}
\newtheorem{proposition}[theorem]{\bf Proposition}
\newtheorem{remark}[theorem]{\bf Remark}


\newcommand{\e}{\varepsilon}

\def\up#1{{}^{#1}}                
\def\dn#1{{}_{#1}}                
\def\du#1#2{{}_{#1}{}^{#2}}       
\def\ud#1#2{{}^{#1}{}_{#2}}       
\def\st#1#2{{}^{#1}_{#2}}         
\def\gt#1#2#3#4{\st{#1\ldots#2}{#3\ldots#4}}

%
\def\half{{\scriptstyle{1 \over 2}}}
\let\leq=\leqslant                
\let\le=\leqslant
\let\geq=\geqslant
\let\ge=\geqslant
\def\Real{{\Bbb R}}               
\def\compact{\subset\subset}      
\let\embed=\iota                  
\def\tR{{\tilde R}}
\def\tg{{\tilde{\mathcal G}}}
\def\tCD{{\tilde{\mathcal D}}}      
\def\CT{{\mathcal T}}               
\def\CK{{\mathcal K}}               
\def\TM{{\mathfrak{X}}}              
\def\tensor{{\mathcal T}}           
%
%
\def\mol{\phi}
\def\sk{\omega}
\def\skk{\omega}
\def\SK{{\mathrm{SK}}}
\def\SO{\Phi}

%
\def\HCA{\hat {\mathcal A}_0}       
\def\CCA{\check{\mathcal A}}        
\newcommand{\atil}{\ensuremath{\tilde{\mathcal{A}}_0(M)} } 
\def\CEM{\hat \CE}              
\def\MEF{\hat\CE}              
\def\CA#1{\tilde{\mathcal A}_{#1}}  
\def\MC{\tilde{\mathcal A}}  
\def\CM{{{\hat{\mathcal E}_M}}}       
\def\CN{\hat{\mathcal N}}               
\def\CGM{\hat{\mathcal G}}          
\def\hCG{\check{\mathcal G}}          
%
%
\def\delo{{\mathop{\nabla}\limits^0}} 
\def\tLie{\tilde  \Lie}         
\def\hLie{\Lie'}            
\def\gLie{\hat \Lie}        
\def\dLie{\Lie}             
\def\Lief{\Lie^{C^\infty}}
\def\Lien{\Lie^{\Omega^n}}
\def\Liesk{\Lie^{\mathrm{SK}}}
\def\Liet{\tilde \Lie}           
\def\pbyp#1#2{{\partial#1\over\partial#2}}
\def\gcov{\hat\nabla}       
\def\gelo{\hat\delo}        
%
%
\def\transp{\Upsilon}           
\def\gls{{\Upsilon_*}\du}       
\def\gus{{\Upsilon^*}\ud}       
\def\met{{\mathcal M}}          

%
%
\def\emph#1{{\it#1}}
\def\x{{\boldsymbol x}}
\def\y{{\boldsymbol y}}
\def\bi{{\boldsymbol i}}
\def\k{{\boldsymbol k}}
\def\l{{\boldsymbol l}}
\def\z{{\boldsymbol z}}

\newcommand{\Ll}{L_{\mbox{\rm\small loc}}}
\newcommand{\Hl}{H_{\mbox{\rm\small loc}}}
\newcommand{\pa}{\partial}



\title{Nonlinear generalised functions on manifolds}

\author{E.~A.~Nigsch}
\address{E.~A.~Nigsch, Institut f\" ur Mathematik, Universit\" at Wien, Vienna, Austria}
\email{eduard.nigsch@univie.ac.at}

\author{J.~A.~Vickers}
\address{J.~A.~Vickers, School of Mathematics, University of Southampton, Southampton SO17 1BJ, UK}
\email{J.A.Vickers@soton.ac.uk}

\subjclass[2010]{46F30, 46T30}

\keywords{nonlinear generalised functions, Colombeau algebra, diffeomorphism invariant}


\begin{abstract}
  This paper lays the foundations for a nonlinear theory of
  differential geometry that is developed in a subsequent paper
  \cite{diffgeom} which is based on Colombeau algebras of tensor
  distributions on manifolds. We adopt a new approach and construct a
  global theory of algebras of generalised functions on manifolds
  based on the concept of smoothing operators. This produces a
  generalisation of previous theories in a form which is suitable for
  applications to differential geometry. The generalised Lie
  derivative is introduced and shown to commute with the embedding of
  distributions. It is also shown that the covariant derivative of a
  generalised scalar field commutes with this embedding at the level
  of association.
\end{abstract}

\maketitle

\section{Introduction}

The classical theory of distributions has proved a very powerful tool
in the analysis of linear partial differential equations. However, the
fact that in general one cannot multiply distributions makes them of
limited use in theories such as general relativity whose underlying
equations are inherently nonlinear. Geroch and Traschen \cite{GT}
identified a class of \emph{regular metrics} for which the components
of the curvature tensor are well defined as distributions and
showed that such regular metrics have curvature with singular support
on a manifold of co-dimension at most one. Thus, one can describe
shells of matter but not strings or particles with metrics in this
class. However, by going outside conventional distribution theory
Colombeau \cite{col1} showed that it is possible to construct
associative, commutative differential algebras which contain the space
of distributions as a linear subspace and the space of smooth
functions as a subalgebra. Colombeau's theory of \emph{generalised
  functions} has therefore increasingly had an important role to play
in general relativity, enabling one to use distributions in situations
where one has ill defined products according to the classical theory,
but without having to resort to ad hoc regularisation
procedures. Applications of Colombeau's theory to general relativity
have included the calculation of nonlinear distributional curvatures
which correspond to metrics of low differentiability, such as those
which occur in space-times with thin cosmic strings \cite{cvw} and
Kerr singularities \cite{balasin}, and the electromagnetic field
tensor of the ultra-relativistic Reissner-Nordstr{\" o}m solution
\cite{steinbauer}. For a review of applications of Colombeau algebras
to general relativity see \cite{SV}.

The basic idea is to represent generalised functions by families of
smooth functions. In the special version of the theory the Colombeau
algebra is denoted $\CG^s$ and the basic space used for its
construction consists of 1-parameter families $(f_\eps)_{\e \in (0,1]}$
of smooth functions. However, this results in many different
representations of what is essentially the same function so that one
identifies families which differ by something \emph{negligible}, i.e.,
by a family of functions whose derivatives vanish faster than any
power of $\eps$ on any compact set. This identification is
  realised by factoring out by the set of such
functions, but this is not an ideal unless one restricts the basic
space to families of \emph{moderate} functions whose derivatives are
bounded on compact sets by some positive power of $1/\eps$. The
(special) algebra of generalised functions is therefore defined to be
moderate functions modulo negligible functions, see \cite{col2} for
more details.

Despite factoring out by negligible functions, the notion of
generalised function within Colombeau algebras is finer than that
within conventional distribution theory, and it is this feature that
enables one to circumvent Schwartz's result on the impossibility of
multiplying distributions \cite{impossibility}. Although the pointwise
product of smooth functions commutes with the embedding into the
algebra, the pointwise product of continuous functions does not (and
indeed this cannot be the case due to the Schwartz impossibility
result). However, an important feature of Colombeau algebras is an
equivalence relation known as \emph{association} which coarse grains
the algebra. At the level of association the pointwise product of
continuous functions does indeed commute with the
embedding. Furthermore, many (but not all) elements of the algebra are
associated to conventional distributions. This feature has the
advantage that in many cases one may use the mathematical power of the
differential algebra to perform classically ill-defined calculations
but then use the notion of association to give a physical
interpretation to the answer.

Unfortunately, the special algebra suffers from the disadvantage that
there is no canonical embedding of distributions into it. In some
situations this is not a problem because some mathematical or physical
feature of the problem may be used to define a preferred
embedding. However, in general there is no such preferred embedding
into the special algebra, so in section 2 we will briefly describe the
{\it full} Colombeau algebra $\CGF$ in which the generalised functions
are parameterised by elements $\mol$ of a space of mollifiers
$\OCA{k}$. This enables one to define an associative commutative
differential algebra on $\Real^n$ which contains the space of smooth
functions as a subalgebra and has a {\it canonical} embedding of the
space of distributions as a linear subspace. Furthermore, the
embedding commutes with (distributional) partial derivatives. Within
$\CGF$ one also has a notion of association which may be used to give
a distributional interpretation to certain generalized functions. This
algebra was used in \cite{cvw} to show that the curvature of a cone is
associated to a multiple of the delta distribution.

Although the full Colombeau algebra on $\Real^n$ permits a canonical
embedding of the space of distributions as a linear subspace,
this has been bought at the price of giving up manifest coordinate
invariance. Indeed, the definition of the spaces $\OCA{k}$ of
mollifiers which are used to define the algebra is coordinate
dependent. One approach to this problem is to regard the use of the
Colombeau algebras as a purely intermediate part of the
construction. For example, in the case of the cone one starts with the
metric in a given coordinate system, calculates the regularised metric
and uses this to calculate the curvature density in $\CGF$. One can
then show that the result is associated to a multiple of the delta
distribution and that furthermore if one repeats the entire
calculation in a different coordinate system the final result is just
the transformed delta distribution (see \cite{VW1} for details).

However, there exist situations in which the generalised functions one obtains are not associated to any distribution and in which it is desirable to have a coordinate invariant generalisation of the full algebra.  Such an algebra was first proposed by Colombeau and Meril \cite{colmeril}.  Their approach was to give a local description of the algebra together with a transformation law for the generalised functions which ensures that the embedding into the algebra commutes with coordinate transformations.  This work suffered from some technical problems but building on these ideas it was shown that one can construct a global Colombeau algebra of generalised functions on manifolds (see \cite{advances} for details) retaining all the distinguishing features of the local theory in the global context. In section 3 we will present a new version of the algebra based on the idea of smoothing operators. This has a larger basic space than \cite{advances} which allows us to define a covariant derivative and can therefore be developed into a nonlinear theory of distributional differential geometry \cite{diffgeom}. In contrast to the theory on $\Real^n$ the theory of generalised functions on manifolds involves a number of technical issues involving in particular the theory of differentiation in locally convex spaces. We will not go into the details here, but the approach will be to use the \emph{convenient setting of global analysis} of \cite{KM}.

For applications of the algebra to general relativity we are interested in Einstein's equations for metrics of low differentiability. These metrics are tensorial rather than scalar objects. Because the embedding into the algebra does not commute with multiplication (except on the subalgebra of smooth functions) one cannot simply work with the coordinate components of a tensor and use the theory of generalised scalars. In a subsequent paper \cite{diffgeom} we show how it is possible to define an algebra of generalised tensor fields on a manifold which contains the spaces of smooth tensor fields as a subalgebra and has a canonical coordinate independent embedding of the spaces of tensor distributions as linear subspaces. 

In order to make the presentation self-contained we begin in this paper by briefly reviewing the Colombeau theory of generalised functions on $\Real^n$ emphasising the structural issues that will be important in generalising this to manifolds.

\section{The full Colombeau algebra on $\Real^n$}

In this section we briefly describe the construction of the full Colombeau algebra in $\Real^n$ (for further details and proofs see \cite{col1}). The starting point is the observation that one can smooth functions by taking the (anti)-convolution with a suitable mollifier. Let $\cD(\Real^n)$ denote the space of smooth functions on $\Real^n$ with compact support. We define $\OCA{0}(\Real^n)$ to be the set of those $\mol \in \cD(\Real^n) $ which satisfy the normalisation condition
\begin{equation*}
\int_{\Real^n} \mol(\x)\, \dx=1.
\end{equation*}
Given $\eps >0$ we set
\begin{equation*}
\mol_{\eps}(\x)=\frac{1}{\eps^n}\mol\left({\x \over \eps}\right), 
\end{equation*}
so that $\mol_\eps$ has support scaled by $\eps$ and its amplitude adjusted so that its integral is still one. 

Note that $(\mol_{\e})_\e$ is an example of a net of smooth functions with the delta distribution as its limit in the sense that
\begin{equation*}
\lim_{\eps \to 0}\int_{\Real^n}\mol_{\eps}(\x)\Psi(\x)\,\dx=\Psi(0) \quad \forall \Psi \in \cD(\Real^n).
\end{equation*}
This is sometimes called a model delta net (see \cite{MO}). Provided $f \in L^1_{\textrm{loc}}$ (i.e., $f$ is a locally integrable function), for each $\phi \in \OCA{0}(\Real^n)$ we can define a 1-parameter family of smooth functions $\tilde f_\eps$ by
\begin{equation}
\tilde f_\eps(\x)=\int_{\Real^n}f(\y)\mol_\eps(\y-\x) \,\dx \label{4}
\end{equation}
which converges to $f$ in $\cD'(\Real^n)$. However, in what follows it will be important to regard $\mol$ as well as $\eps$ as a parameter so we write expression \eqref{4} as $\widetilde f(\mol_\eps, \x)$.

It will also be convenient to introduce the {\it translation operator} $\tau$ defined by
\begin{equation*}
\left(\tau_{\x}\mol\right)(\y)=\mol(\y-\x) 
\end{equation*}
for $x,y \in \Real^n$ and $\mol \in \cD(\Real^n)$. In order to match the notation of the theory on manifolds we will sometimes write $\mol_{\x,\eps}=\tau_{\x}\mol_{\eps}$, so that for fixed $\x$, $\mol_{\x,\eps}$ is a 1-parameter family of smooth functions converging in $\cD'(\Real^n)$ to $\delta_{\x}$, the delta distribution at $\x$.

A distribution $T \in \cD'(\Real^n)$ is a linear functional on the space of smooth test functions $\cD(\Real^n)$ and we may generalise equation \eqref{4}  to distributions by defining a 1-parameter family of smooth functions $\tilde T_\e$ by
\begin{equation}
\tilde T_\e=\langle T, \tau_{\x}\mol_\eps \rangle.
\end{equation}
Again, we will write this expression as $\tilde T(\phi_\e, \x)$.

In order to construct the algebra of generalised functions we define a grading on the space of mollifiers in terms of moment conditions. Note that we will throughout use multi-index notation so that $\bi=(i_1, \dots i_n)$ and $\x^\bi=x_1^{i_1}\dots x_n^{i_n}$.

\begin{definition}
For $q \in \Nat$ we define $\OCA{q}(\Real^n)$ to be the set of functions $\mol \in \OCA{0}(\Real^n)$ such that
\begin{equation*}
\int_{\Real^n} \x^\bi \mol(\x)\, \dx=0 \quad \forall \bi \in \Nat^n_0 \textrm{ with } \abso{\bi} \leq q.  
\end{equation*}
\end{definition}

We are now in a position to construct the full Colombeau algebra on $\Real^n$. Our basic space will be the following.
\begin{definition}
$\CEF(\Real^n)$ is defined to be the set of all maps
\begin{align*}
F \colon \OCA{0}(\Real^n) \times \Real^n &\to \Real \\
(\mol, \x) &\mapsto F(\mol, \x) 
\end{align*}
which for fixed $\mol$ are smooth as functions of $\x$.
\end{definition}

The lack of any continuity requirement with respect to $\mol$ reflects their role as parameters rather than test functions.

On $\CEF(\Real^n)$ we may define the product $FG$ by
\begin{equation*}
(FG)(\mol, \x)=F(\mol, \x)G(\mol, \x) 
\end{equation*}
and the derivative operation
\begin{equation*}
(\partial_iF)(\mol, \x)={\partial \over {\partial x^i}}\left(F(\mol, \x)\right) 
\end{equation*}
for $i=1 \dotsc n$, which together give $\CEF(\Real^n)$ the structure of a differential algebra.

However, as it stands, the space $\CEF(\Real^n)$ is much too large and, thinking in terms of the limit $\e \to 0$, contains many representations of what are essentially the same functions. For example, to represent a given smooth function $f \in C^\infty(\Real^n)$ we may define $\tilde f \in \CEF(\Real^n)$ by
\begin{equation}
\tilde f(\mol_\e, \x)=\int_{\Real^n} f(\y) \mol_\e(\y-\x) \, \dy \label{8}
\end{equation}
but since $f$ is smooth we can also define another family $\hat f(\mol, \x)$ (which does not in fact depend on $\mol$) by
\begin{equation}
\hat f(\mol_\e, \x)=f(\x). \label{9}
\end{equation}
Note that in the above equations \eqref{8} and \eqref{9} we have chosen to use the scaled mollifiers $\mol_\eps$. Strictly speaking, however, when using these equations to define elements of $\CEF(\Real^n)$ one uses a general mollifier $\phi \in \OCA{0}(\Real^n)$ (see below for details).

We therefore want to introduce an equivalence relation such that $\tilde f$ and $\hat f$ (and their derivatives) become equivalent. Expanding $\tilde f(\mol_\eps, \x)$ in a Taylor series and using the moment conditions for $\mol \in \OCA{k}(\Real^n)$ we see that
\begin{align*}
   \tilde f(\mol_\eps,\x) &= {1\over\eps^n} \int_{\Real^n} f(\y)
   \mol\left({\y-\x\over\eps}\right) \,d\y \\
   &= \int_{\Real^n} f(\x+\eps \y) \mol(\y) \, d\y \\
   &= \int_{\Real^n} \left\{ f(\x) + \sum_{\abso{\l}=1}^q {\eps^{\abso{\l}}\y^\l\over
{\abso{\l}!}}\partial^\l f(\x) \right.\\
   &\quad +\left. \sum_{\abso{\l} = q+1} \frac{q+1}{\l!} \e^{q+1} \y^\l \int_0^1 (1-t)^q \partial^\l f(x+t\e y)\,\rd t \right\} \phi(\y) \rd\y \\
   &= f(\x) + \eps^{q+1} \sum_{\abso{\l} = q+1} \frac{q+1}{\l!} \int_{\Real^n} \int_0^1 \y^\l
   (1-t)^q \cdot \\
   &\qquad \qquad \cdot \partial^\l f(\x+t \eps \y) \phi(\y) \,\rd t\, \rd\y \\
   &= \hat f(\mol_\eps, \x) + O(\eps^{q+1})
\end{align*}
where $\partial^\l$ is the derivative operator given by $\partial^\l=\partial_1^{l_1}\dots \partial_n^{l_n}$. Thus, by choosing $\mol$ to be in $\OCA{k}(\Real^n)$ for suitably large $k$ we can make $\tilde f - \hat f$ tend to zero like an arbitrary power of $\eps$. Requiring a similar condition for the derivatives motivates the following definition.

\begin{definition}[Negligible functions]\label{def_negl}
$\CNF(\Real^n)$ is defined to be the set of functions $F \in \CEF(\Real^n)$ such that for all compact $K \subset \Real^n$, for all $\k \in \Nat^n_0$ and for all $m \in \Nat$, there is some $q \in \Nat$ such that if $\mol \in \OCA{q}(\Real^n)$ then 
\begin{equation*}
\sup_{\x \in K}\abso{\partial^\k F(\mol_\eps, \x)} = O(\eps^m) 
\quad\textrm{as }\eps \to 0. 
\end{equation*}
\end{definition}
Note that the derivative $\partial^\k$ acts only on the $x$-variable here, contrary to the situation later on where we also have to consider derivatives with respect to $\phi$.

The key result that follows from this is that for a smooth function $f$ we have that $\tilde f -\hat f$ is in $\CNF(\Real^n)$. However, in order to define an algebra we would like to factor out by $\CNF(\Real^n)$, and this requires it to be an ideal. Unfortunately, this is not the case because we can multiply elements of $\CNF(\Real^n)$ by elements of $\CEF(\Real^n)$ with rapid non-polynomial growth in $1/\eps$ so that the conditions of Definition \ref{def_negl} are no longer satisfied. We therefore restrict $\CEF(\Real^n)$ to the subalgebra of functions of moderate growth in the following sense.
\begin{definition}
$\CMF(\Real^n)$ is defined to be the set of functions $F \in \CEF(\Real^n)$ such that for all compact $K \subset \Real^n$, for all $\k \in \Nat^n_0$, there is some $N \in \Nat$ such that if $\mol \in \OCA{N}(\Real^n)$ then
\begin{equation*}
\sup_{\x \in K}\abso{\partial^\k F(\mol_\eps, \x)} = O(\eps^{-N})
\quad\textrm{as }\eps \to 0.  
\end{equation*}
\end{definition}

\begin{proposition}
$\CNF(\Real^n)$ is an ideal in $\CMF(\Real^n)$. 
\end{proposition}

We may therefore define the space of generalised functions $\CGF(\Real^n)$ as a factor algebra.

\begin{definition}[Generalised functions]
\begin{equation*}
\CGF(\Real^n)=\CMF(\Real^n)/\CNF(\Real^n).
\end{equation*}
\end{definition}

Although the definition of a negligible function $F$ requires estimates for the derivatives $\abso{\partial^\k F(\mol_\eps, \x)}$ these are in fact not needed as is shown by the following useful proposition.
\begin{proposition}\label{noderiv}
Let $F \in \CMF(\Real^n)$ be such that for all compact $K \subset \Real^n$, for all $m \in \Nat$, there is some $q \in \Nat$ such that if $\mol \in \OCA{q}(\Real^n)$ then 
\begin{equation*}
\sup_{\x \in K}\abso{F(\mol_\eps, \x)} = O(\eps^m) \quad\hbox{as $\eps \to 0$}. 
\end{equation*}
Then $F \in \CNF(\Real^n)$.
\end{proposition}
Hence any moderate function which satisfies the negligibility condition, without differentiating, is negligible. For the proof see \cite[Theorem 1.4.8]{book}.

One may now show that one has an embedding 
\begin{align*}
\iota \colon \cD'(\Real^n) &\to \CGF(\Real^n) \\
T &\mapsto [\tilde T]
\end{align*}
where $[\tilde T]$ denotes the equivalence class of $\tilde T \in \CMF(\Real^n)$ and $\tilde T(\mol, \x) \coleq \langle T, \tau_{\x}\mol \rangle$. The only thing we need to establish is that $\tilde T$ is moderate.

As we may assume without limitation of generality that $\tilde T = \partial^\l f$ for some continuous function $f(\y)$ in a neighborhood of a given compact set, differentiating the expression for $\tilde T$ with respect to $\x$ we obtain
\[ \partial^\k \tilde T(\mol_\e, \x) = \langle \partial^\l_\y f, \partial^\k_\x ( \tau_x \mol_\e ) \rangle = \langle f, (-1)^{\abso{\l}} \partial_\y^\l \partial_\x^\k (\tau_\x \mol_\e) \rangle. \]
Now $\tau_\x \mol_\e(\y) = \frac{1}{\e^n} \mol\left(\frac{\y-\x}{\e}\right)$ so that
\[ (-1)^{\abso{\l}} \partial_\y^\l \partial_\x^\k ( \tau_\x \mol_\e(\y)) = \frac{1}{\e^{n+\abso{\l} + \abso{\k}}} (-1)^{\abso{\l} + \abso{\k}} \mol^{(\k+\l)}\left(\frac{\y-\x}{\e}\right). \]
Thus, uniformly for $x$ in a compact set we have
\begin{align*}
\partial^\k \tilde T(\mol_\e, \x) & = \frac{1}{\e^{n+\abso{\l} + \abso{\k}}} \int f(y) (-1)^{\abso{\l}+\abso{\k}} \mol^{(\k+\l)} \left(\frac{\y-\x}{\e}\right)\,\dy \\
& = \frac{1}{\e^{\abso{\l} + \abso{\k}}} \int f(\x + \e\y) (-1)^{\abso{\l} + \abso{\k}} \mol^{(\k+\l)}(\y)\,\dy = O(\e^{-\abso{\l} - \abso{\k}})
\end{align*}
so that $\tilde T$ is moderate.

The main properties of $\CGF(\Real^n)$ are contained in the following proposition. For proofs and further details see \cite{col1} and  \cite{MO}.
\begin{proposition}
\leavevmode
\begin{enumerate}[label=(\alph*)]
  \item $\CGF(\Real^n)$ is an associative commutative differential algebra.
  \item The embedding $\iota$ defined above embeds $\cD'(\Real^n)$ as a linear subspace.
  \item For smooth functions $f \in C^{\infty}(\Real^n)$ we have $\iota(f)=[\hat f]$ where $\hat f(\mol, \x)=f(\x)$, so that $\CGF(\Real^n)$ contains the space of smooth functions as a subalgebra.
  \item The embedding commutes with (distributional) partial differentiation so that
\[ \iota(\partial^\k T)(\mol, \x)=\partial^\k(\iota T)(\mol,\x). \]
 \end{enumerate}
\end{proposition}

As we remarked earlier an important concept is that of association.
\begin{definition}[Association]
We say an element $[F]$ of $\CGF(\Real^n)$ is associated to $0$ (denoted $[F] \approx 0$) if for each $\Psi \in \cD(\Real^n)$ there exists some $p>0$ with
\begin{equation*}
\lim_{\eps \to 0}\int_{\x \in \Real^n}F(\mol_\eps, \x)\Psi(\x)d\x=0 \quad \forall \mol \in \OCA{p}(\Real^n). 
\end{equation*}
We say two elements $[F],[G]$ are associated and write $[F]\approx [G]$ if $[F-G] \approx 0$.
\end{definition}

\begin{definition}[Associated distribution]
We say $[F] \in \CGF(\Real^n)$ admits $T \in \cD'(\Real^n)$ as an associated distribution if for each $\Psi \in \cD(\Real^n)$ there exists some $p>0$ with
\begin{equation*}
\lim_{\eps \to 0}\int_{\x \in \Real^n}F(\mol_\eps, \x)\Psi(\x)d\x=\langle
T, \Psi \rangle \quad
\forall \mol \in \OCA{p}(\Real^n). 
\end{equation*}
\end{definition}

These definitions do not depend on the choice of representative; moreover, note that not all generalised functions are associated to a distribution.

At the level of association we regain the following compatibility results for multiplication of distributions.

\begin{proposition}\label{prop2.10}
\leavevmode
\begin{enumerate}[label=(\alph*)]
 \item If $f \in C^{\infty}(\Real^n)$ and $T \in \cD'(\Real^n)$ then 
\begin{equation*}
\iota(f)\iota(T) \approx \iota(fT). 
\end{equation*}
\item If $f, g \in C^{0}(\Real^n)$ then 
\begin{equation*}
\iota(f)\iota(g) \approx \iota(fg).  
\end{equation*}
\end{enumerate}
\end{proposition}

Although the partial derivative commutes with the embedding this is not true of the Lie derivative. Let $X(\x) \in \TM(\Real^n)$ be a smooth vector field on $\Real^n$ and $T \in \mathcal{D}'(\Real^n)$, then 
\begin{align*}
(\Lie_X\tilde T)(\mol_\eps,\x)&=X^a(\x)\partial_a(\tilde T(\mol_\eps,\x)) \\
&=X^a(\x)\langle \partial_a T, \mol_{\x,\eps}\rangle.
\end{align*}
On the other hand,
\begin{equation*}
(\widetilde{\Lie_X T})(\mol_\eps,\x)= \langle X^a\partial_aT,
\mol_{\x,\eps} \rangle
\end{equation*}
These two expressions are not the same in general since the first only involves the value of the vector field at $\x$, while the second involves the values in a neighbourhood of $\x$. In fact, if these expressions always were the same this would mean the embedding commutes with multiplication by smooth functions, which contradicts the Schwartz impossibility result. However, by part (a) of Proposition \ref{prop2.10} the two expressions are associated since $X^a$ is a smooth function for $a=1 \dots n$. We also note that if $f$ is a smooth function then $\Lie_X\tilde f=\widetilde{\Lie_X f}$ since we may represent $\tilde f$ by $f$ and $\widetilde{\Lie_X f}$ by $\Lie_X f$. We therefore have the following proposition.
\begin{proposition} 
\leavevmode
\begin{enumerate}[label=(\alph*)]
 \item Let $f \in C^\infty(\Real^n)$ and $X$ be a smooth vector
field. Then,
\begin{equation}
\Lie_X (\iota(T)) = \iota(\Lie_X T).  
\end{equation}
\item Let $T \in \cD'(\Real^n)$ and $X$ be a smooth vector field. 
Then,
\begin{equation}
\Lie_X( \iota(T)) \approx \iota(\Lie_X T). 
\end{equation}
\end{enumerate}
\end{proposition}

It is also possible to localise the entire construction to obtain $\CGF(\Omega)$ for open sets $\Omega \subset \Real^n$ by restricting $\x$ to lie in $\Omega$ in the relevant definitions. The only technical complication relates to the embedding where one must first extend the distribution and then show that the result is independent of the extension (see \cite{col1} for details).

\section{Smoothing distributions and the Colombeau algebra on manifolds}

A coordinate independent description of generalised functions on open
sets $\Omega \subset \Real^n$ was proposed by Colombeau and Meril
\cite{colmeril}. However, this suffered from a number of minor
defects; in particular, the definition of $\OCA{k}(\Omega)$ did not
take into account the $x$-dependence of the mollifiers which meant
that the definition of moderate functions was dependent on the
coordinate system used. An explicit counterexample due to Jel\'{\i}nek
\cite{jelinek} demonstrated that the construction was not in fact
diffeomorphism invariant.  In the same paper Jel\'{\i}nek gave an
improved version of the theory which clarified a number of important
issues but fell short of proving the existence of a coordinate
invariant algebra.  The existence of a (local) diffeomorphism
invariant Colombeau algebra on open subsets $\Omega$ of $\Real^n$ was
finally established Grosser et al.\ \cite{mem}.  In parallel with
this, \cite{VW2} proposed a definition of a global manifestly
diffeomorphism invariant theory on manifolds. By making use of the
characterisation results of \cite{mem} it was shown in \cite{advances}
that one can construct a global Colombeau algebra $\CG(M)$ of
generalised functions on manifolds. There, it was demonstrated how to
obtain a canonical linear embedding of $\cD'(M)$ into $\CG(M)$ that
renders $C^\infty(M)$ a faithful subalgebra of $\CG(M)$. In addition,
it was shown that this embedding commutes with the generalised Lie
derivative, ensuring that the theory retains all the distinguishing
features of the local theory in the global context. Although this
theory has a well defined generalised Lie derivative it turns out that
there is no natural definition of a generalised covariant
derivative. In this section we describe a new approach to Colombeau
algebras \cite{Nigsch} based on the concept of smoothing operators
that it is closer to the intuitive idea of a generalised function as a
family of smooth functions. This results in a new basic space which
allows us to define both a generalised Lie derivative and a covariant
derivative. Replacing the spaces $\mathcal{A}_k(\Real^n)$ by suitable spaces
of smoothing kernels we are able to use asymptotic versions of the moment conditions
and hence do not need such a grading anymore, which results in a quantifier less
in the definitions of moderateness, negligiblity and association.
In contrast to \cite{advances}, which made use of the
local theory in a number of key places, in the current paper we give
intrinsic definitions on the whole of the manifold $M$. As
here we only outline the general theory, we refer for full proofs to
\cite{distcurv}.

On $\Real^n$ the space of distributions $\cD'(\Real^n)$ is dual to the space of smooth functions of compact support, whereas on an orientable manifold the space of distributions $\cD'(M)$ is dual to $\Omega^n_c(M)$, the space of $n$-forms of compact support (note that on not necessarily orientable manifolds, one uses densities instead of $n$-forms; on an oriented manifold these are the same). In the Colombeau theory on $\Real^n$ smoothness of the embedded functions is obtained by integrating against mollifier functions $\mol(\y-\x)$. The obvious generalisation on manifolds is to replace the function by an $n$-form $\omega$. However, on a manifold it does not make sense to look at $\omega(y-x)$ since $y-x$ has no coordinate independent meaning, so instead we will look at objects $\sk_x(y)$ which are $n$-forms in $y$ parameterised by $x \in M$. We therefore make the following definition.

\begin{definition}
A \emph{smoothing kernel} $\sk$ is a smooth map
\begin{eqnarray*}
\sk \colon M &\to& \Omega^n_c(M)\\
x &\mapsto& \sk_x
\end{eqnarray*} 
and we denote the space of such objects $\SK(M)$. Thus $\SK(M)=C^\infty(M, \Omega^n_c(M))$.
\end{definition}
The key new idea (see \cite{Nigsch} for more details) is not to immediately try and generalise the Colombeau construction of $\Real^n$ to a manifold $M$ in some ad-hoc way, but to start with the notion of smoothing operator.

\begin{definition}
A {\it smoothing operator} $\SO$ on $M$ is a linear continuous map
\begin{equation*}
\SO \colon \cD'(M) \to C^\infty(M)
\end{equation*}
We denote the space of such objects by $L(\cD'(M),C^\infty(M))$.
\end{definition}

Given a smoothing operator $\SO$ we may associate to it a smoothing kernel $\sk$ in the following way: if for $u \in \cD'(M)$ and $x \in M$ we demand that
\[ \SO(u)(x) = \langle u, \sk_x \rangle \qquad \forall u \in \cD'(M),\ \forall x \in M \]
then this implies
\[ \sk_x(y) = \langle \delta_y, \sk_x \rangle = \SO(\delta_y)(x). \]
Thus given a smoothing operator $\SO \in L(\cD'(M),C^\infty(M))$, this converts a distribution $u \in \cD'(M)$ into a smooth function $\SO(u)$ by the action of $u$ on the smoothing kernel $\sk$ where $\sk_x(y)\coleq\SO(\delta_y)(x)$. Conversely, given a smoothing kernel $\sk \in \SK(M)$ we obtain a smoothing operator by $\SO(u)(x)\coleq <u, \sk_x>$. Indeed, it follows from a variant of the Schwartz kernel theorem that this correspondence is a
topological isomorphism
\begin{equation}
L_b(\cD'(M), C^\infty(M)) \cong C^\infty(M, \Omega^n_c(M))=\SK(M)
\label{kernel}
\end{equation}
where the left hand side has the topology of bounded convergence and the right hand side has the topology of uniform convergence on compact sets in all derivatives.

We therefore take our basic space $\MEF(M)$ of generalised functions to consist of (smooth) maps from the space of smoothing kernels to the space of smooth functions,
\[ \MEF(M)\coleq C^\infty(\SK(M), C^\infty(M)). \]
Note that in this definition (and elsewhere in the paper where we consider smooth maps between infinite dimensional spaces) we will use the definition of smoothness based on \emph{the convenient setting of global analysis} of \cite{KM}. The basic idea of this approach is that a map $f \colon E\to F$ between locally convex spaces is smooth if it transports smooth curves in $E$ to smooth curves in $F$ (where the notion of smooth curves is straightforward via limits of difference
quotients).

Actually the basic space $\MEF(M)$ is somewhat larger than we would want since it allows $F(\sk)$ to depend on $\sk$ globally, which destroys the sheaf character of the algebra. We therefore restrict to a sub-algebra $\MEF_{\mathrm{loc}}(M)$ consisting of \emph{local} elements $F \in \MEF(M)$, defined by the property that if two smoothing kernels $\sk$ and $\tilde \sk$ agree on some open set $U$ then $F(\sk)$ and $F(\tilde \sk)$ also agree on $U$. Note that all embedded elements satisfy this condition so that there is no real loss of generality in restricting to this space. Therefore, for the rest of the paper we will work exclusively with $\MEF_{\mathrm{loc}}(M)$ but for ease of notation we will simply write it as $\MEF(M)$. For an in-depth exposition of this topic we refer to \cite{specfull}.

The basic space naturally contains both $\cD'(M)$ and $C^\infty(M)$ via the linear embeddings $\iota$ and $\sigma$
\begin{equation*}
\begin{aligned}
\iota \colon \cD'(M) &\to \MEF(M) \qquad & (\iota u)(\sk)(x) & \coleq <u, \sk_x> \\
\sigma \colon C^\infty(M) &\to \MEF(M) \qquad & (\sigma f)(\sk)(x) & \coleq f(x)
\end{aligned}
\end{equation*}
and inherits the algebra structure from $C^\infty(M)$ through the product
\begin{equation*}
(F_1\cdot F_2)(\sk)\coleq F_1(\sk)F_2(\sk), \qquad F_1,F_2 \in \MEF(M),\ \sk \in \SK(M).
\end{equation*}

We may regard a smooth function as a regular distribution so that one may embed it either via $\sigma$ to obtain $(\sigma f)(\sk)(x)=f(x)$ or via $\iota$ to obtain $(\iota f)(\sk)(x)=\int f(y) \sk_x(y)$. In order to identify these expressions we would like to set $\sk_x=\delta_x$. Strictly speaking this is not possible, but replacing $\sk_x$ by a net $(\sk_{x,\e})_\e$ of $n$-forms which tends to $\delta_x$ appropriately as $\epsilon \to 0$ and using suitable asymptotic estimates to define negligibility allows us to construct a quotient algebra in which the two embeddings of smooth functions agree.

The next key concept required is therefore that of a delta net of smoothing kernels $\sk_\eps$ which will play the role of the $\eps$ dependent mollifiers $\mol_{x,\eps}$ used in the embedding of distributions on $\Real^n$. Since we are working on a manifold we do not have translation and scaling operators available, so we need to consider carefully what properties are required. Again, rather than simply trying to copy the construction on $\Real^n$ it is useful to look at what is required from the point of view of the corresponding family of smoothing operators. The key properties are that:
\begin{enumerate}[label=(\alph*)]
 \item\label{prop1} the family of smoothing operators should be localising,
 \item\label{prop2} in the limit the smoothing operator when applied to smooth functions should be the identity in $C^\infty(M)$,
 \item\label{prop3} the family of smoothing operators should satisfy some seminorm estimates which control the growth and
 \item\label{prop4} in the limit the smoothing of a distribution $u$ should converge in $\cD'(M)$ to $u$.
\end{enumerate}
Property \ref{prop1} ensures that the support of the corresponding net of smoothing kernels shrinks, \ref{prop2} ensures that (in the quotient algebra) the embeddings $\sigma$ and $\iota$ coincide, \ref{prop3} ensures that the embedding of distributions is moderate and property \ref{prop4} shows that an embedded distribution is associated to the original distribution. More precisely, given a family of smoothing operators $(\SO_\eps)_{\eps \in (0,1]}$ we require

\begin{enumerate}[label=(\alph*)]
\item on any compact $K \subset M$ $\forall r>0$ $\exists \eps_0>0$ $\forall x\in K$ $\forall \eps \le \eps_0$ $\forall u \in \cD'(M)$:
  \[ \left(u|_{B_r(x)}=0 \Rightarrow \SO_\eps(u)(x)=0\right); \]
\item for any continuous seminorm $p$ on $L_b(C^\infty(M), C^\infty(M))$ and all $m \in \Nat$ we have
\[ p( \SO_{\eps}|_{C^\infty(M)} - \id ) = O(\e^m); \]
\item for any continuous seminorm $p$ on $L_b(\cD'(M), C^\infty(M))$ there is  $N \in \Nat$ such that
\[ p(\SO_\eps) = O(\eps^{-N}); \]
\item $\SO_{\eps} \to \id$ in $L_b(\cD'(M), \cD'(M))$.
\end{enumerate}

Note that in the second condition we demand convergence like $O(\e^m)$ for all $m$ at once, contrary to Colombeau's original algebra presented above.

We now use the topological isomorphism \eqref{kernel} to translate these conditions into conditions on a net $(\sk_\e)_\e$ of smoothing kernels. The first translates into the requirement that the support of the net shrinks, or more precisely that
\begin{gather*}
\textrm{on any compact }K \subset M\ \forall r>0\ \exists \eps_0>0\\
\forall x\in K\ \forall \eps \le \eps_0: \supp \sk_{x,\eps}\subseteq B_r(x).
\end{gather*}
To do this we need to introduce a Riemannian metric $h$ on $M$ in order to measure the radius of the ball. However, it is not hard to see that the condition does not depend upon the particular choice of Riemannian metric.

To formulate the next condition we need the Lie derivative of a smoothing kernel $\sk$, which we will introduce in terms of the 1-parameter family of diffeomorphisms induced by a vector field. In principle we can consider two different diffeomorphisms $\mu$ and $\nu$ which act separately on the $x$ and $y$ variables of $\sk$, i.e., the pullback action on the parameter $x$ (for fixed $y$) given by $(\mu^*\sk)_x\coleq \sk_{\mu(x)}$ on the one hand and the pullback action on the form (for fixed $x$) given by $\nu^*(\sk_{x})$ on the other hand. We will denote the combined pullback action on the smoothing kernel by $(\mu^*, \nu^*)\sk \coleq \nu^*(\sk_{\mu(x)})$.

We can therefore also consider two different (complete) vector fields $X$ and $Y$ with corresponding flows $\Fl^X_t$ and $\Fl^Y_t$ acting on the $x$ and $y$ variables. This enables us to define the (double) Lie derivative 
\begin{equation*}
\dLie_{(X,Y)}\sk=\left.\frac{d}{dt}\right|_{t=0}
\left(({\mathrm{Fl}}^{X})^*_t,({\mathrm{Fl}}^{Y})^*_t\right)\sk.  
\end{equation*}
Varying the $x$ and $y$ variables separately we have two Lie derivatives
\begin{equation*}
(\dLie_{(X,0)}\sk)_x= \left.\frac{d}{dt}\right|_{t=0}\sk_{\mathrm{Fl}^X_t(x)}
\end{equation*}
and
\begin{equation*}
(\dLie_{(0,Y)}\sk)_x=\left.\frac{d}{dt}\right|_{t=0}
({\mathrm{Fl}}^{Y})^*_t(\sk_x) 
\end{equation*}
and hence
\begin{equation*}
\dLie_{(X,Y)}\sk=\dLie_{(X,0)}\sk+\dLie_{(0,Y)}\sk.
\end{equation*}
Since $\dLie_{(X,0)}\sk$ is given by the formula for the Lie derivative of a function we will denote this derivative by $\Lief_X\sk$, and since $\dLie_{(0,Y)}\sk$ is given by the Lie derivative of an $n$-form we will denote this derivative by $\Lien_Y\sk$. Finally, we will often want to take the geometrically natural Lie derivative $\dLie_{(X,X)}\sk$ of a smoothing kernel which we denote $\Liesk_X\sk$. Note that $\dLie_{(X,0)}\sk$ is denoted $L'_X\sk$ and $\dLie_{(0,Y)}\sk$ is denoted $L_Y\sk$ in \cite{advances}.

We may now define the convergence corresponding to the second
condition above by demanding that for all compact subsets $K \subset M$,
all $m \in \Nat_0$ and all smooth vector fields $X_1, \dotsc, X_m$ on
$M$ we have
\begin{equation*}
\sup_{x \in K}\abso{\left(\int_M
f\Lief_{X_1}\dots\Lief_{X_m}\sk_{x,\eps}\right)-\Lie_{X_1}\dots\Lie_{X_m}f(x)}
=O(\eps^m)
\label{moment}
\end{equation*}
as $\e \to 0$.

It turns out that the third condition, which allows us to establish the fact that the embedding of a distribution is moderate, takes the following form. For any distribution $u \in \cD'(M)$, on any compact subset $K \subset M$ we require
$\forall k \in \Nat_0$ $\forall X_1,\dots,X_k \in\mathfrak{X}(M)$ $\exists N \in \Nat$ such that 
\begin{equation*}
\sup_{x\in K} \abso{\Lie_{X_1}\dots \Lie_{X_k} \langle u, \sk_{x,\eps} \rangle } 
= O(\eps^{-N}).
\end{equation*}

Finally, the fourth condition which gives convergence in $\cD'(M)$ is given by the condition that $\forall u \in \cD'(M)$, $\forall \omega \in \Omega^n_c(M)$
\begin{equation*}
\lim_{\eps \to 0}\int_{x \in M}\langle u, \sk_{\eps,x} \rangle \omega(x)=\langle
u, \omega \rangle.
\end{equation*}

We are now in a position to define a delta net of smoothing kernels (cf.~\cite{distcurv} where the corresponding nets are called \emph{test objects}).

\begin{definition}[Delta Nets of Smoothing kernels] \label{kernels}
$(\sk_\e)_\e \in \SK(M)^{(0,1]}$ is called a delta net of smoothing kernels if on any compact subset $K$ of $M$ it satisfies the following conditions:
\begin{enumerate}[label=(\arabic*)]
\item $\forall r>0$  $\exists \eps_0$ $\forall x\in K$ $\forall \eps \le \eps_0$: $\supp \sk_{x,\eps}\subseteq B_r(x)$;
\item $\forall m \in \Nat$, as $\e \to 0$: \[ \sup_{x \in K}\abso{\left(\int_M f\Lief_{X_1}\dots\Lief_{X_k}\sk_{x,\eps}\right)-\Lie_{X_1}\dots\Lie_{X_k}f(x)} =O(\eps^m);\]
\item $\forall u \in \cD'(M)$ $\forall k \in \Nat_0$ $\exists N \in \Nat$ $\forall X_1,\dots,X_k \in\mathfrak{X}(M)$:
\[ \sup_{x\in K} \abso{\Lie_{X_1}\dots \Lie_{X_k} \langle u, \sk_{x,\eps} \rangle } = O(\eps^{-N}); \]
\item $\forall u \in \cD'(M)$ $\forall \omega \in \Omega^n_c(M)$:
\[ \lim_{\eps \to 0}\int_{x \in \Real^n}\langle u, \sk_{\eps,x} \rangle \omega(x)=\langle u, \omega \rangle. \]
\end{enumerate}

The space of delta nets smoothing kernels on $M$ is denoted $\MC(M)$.
\end{definition}

\begin{remark}
We have seen in the previous section that the moment conditions on $\Real^n$ allow one to show that for a smooth function $f$ and for $\mol \in \OCA{q}(\Real^n)$ we have (in the case $n=1$)
\begin{equation}
\tilde f(\mol_\eps,x)= f(x) + 
{\eps^{q+1}\over {q!}} \int_{\Real} \int_0^1 y^{q+1} (1-t)^q
   f^{(q+1)}(x+t \eps y) \mol(y) \,\rd t \,\rd y 
\end{equation}
so that
\begin{equation}
\tilde f(\mol_\eps, x)= f(x) +O(\eps^{q+1}) \label{moment}
\end{equation}
with a similar argument giving the same estimate for the derivatives. By Proposition \ref{noderiv} this shows that $\tilde f -\hat f$ is negligible and hence that the two possible embeddings of a smooth function coincide in the algebra. On a manifold we have turned things round and instead used \eqref{moment} to characterise the moment condition. As is the case in $\Real^n$ we will use this condition to show that the two possible embeddings of smooth functions differ by a negligible function and hence coincide in the factor algebra.
\end{remark}

Although not necessary for the bare constrution of the theory, it is beneficial for practical calculations to add $L^1$-conditions on the    nets of smoothing kernels. For example, if we also require that 
\begin{equation} \label{l1cond}
\int_M \abso{\sk_{x,\eps}} \to 1 \quad \textrm{ uniformly for }x\textrm{ in compact subsets of }M
\end{equation}
so that (asymptotically) the $L^1$-norm of the smoothing kernels is unity, one can then show that
\begin{equation*}
\lim_{\eps \to 0}\int_{M}f \sk_{x,\eps}
= f(x) \quad\forall f \in C^0(M). 
\end{equation*}
Hence $\iota(f)(\sk_\eps)=<f, \sk_\eps>$ converges to $f$
pointwise. However, this condition is different from condition (2) which involves convergence in $C^\infty(M)$ and requires that the derivatives (of arbitrary order) also converge to the derivatives of $f$.
Another useful condition imitating the behaviour of scaled and translated mollifiers is
\begin{equation}\label{l1cond2}
\forall K \subseteq M \textrm{ compact}\ \forall \alpha \in \Nat_0^n: \sup_{x \in K} \int \abso{\partial_x^\alpha \sk_\e(x)} = O(\e^{-\abso{\alpha}}).
\end{equation}

Before turning to the definition of moderate and negligible functions
we consider the definition of the Lie derivative for elements of the
basic space. There are two different ways of thinking about the Lie
derivative of an element $F \in \MEF(M)$. The first comes from looking
at the pullback action of the diffeomorphism group on the basic space
(which we call the geometrical or generalised Lie derivative) while
the second comes from thinking of $F(\sk)$ for fixed $\sk$ as a smooth
function.  The former has the advantage that it commutes with the
embedding of distributions, but on the other hand it cannot be
$C^\infty$ linear in $X$ (since having {\it both} properties would
violate the Schwartz impossibility result). The latter is simply the
ordinary Lie derivative of a smooth function and therefore agrees with
the directional derivative or covariant derivative of a function. This
will allow us to define the covariant derivative of a generalised
tensor field as in \cite{diffgeom}. Although the ordinary Lie
derivative does not commute with the embedding of distributions, as is
the case on $\Real^n$, it does so at the level of association.

To consider the geometric Lie derivative we start by looking at the action of a diffeomorphism on a generalised function.
\begin{definition}[Pullback action]
If $\psi:M \to N$ is a diffeomorphism then we define the pullback $\psi^*: \MEF(N) \to \MEF(M)$ by
\[ (\psi^* F)(\sk)(x)\coleq  F\bigl(((\psi^{-1})^*,(\psi^{-1})^*)\sk\bigr)(\psi(x)). \]
\end{definition}

We are now in a position to define the Lie derivative.
\begin{definition}[Geometrical Lie derivative] 
Let ${\mathrm{Fl}}^X_t$ be the flow generated by the (complete) smooth vector field $X$. Then for $F \in \MEF(M)$ we set
\begin{equation*}
\gLie_XF=\left.\frac{d}{dt}\right|_{t=0} (\Fl^X_t)^* F.
\end{equation*}
\end{definition}
Using the chain rule we may write this as
\begin{equation*}
 (\gLie_XF)(\sk)=-dF(\sk)(\Liesk_X\sk)+\Lie_X(F(\sk))
\end{equation*}
and since this formula may also be applied to a non-complete vector field we take this as the definition in the general case.
\begin{definition}[Generalised Lie Derivative)] 
For any $F\in \MEF(M)$ and any $X\in \mathfrak{X}(M)$ we set
\begin{equation} \label{liealg}
(\gLie_X F)(\sk) \coleq -dF(\sk)(\Liesk_X\sk)+\Lie_X(F(\sk))
\end{equation}
\end{definition}
\begin{remark}\label{thesame}
In the terminology of \cite{specfull}, the basic space of \cite{advances} is given by the $(\omega_x,x)$-local elements of $\MEF(M)$. On these, the formula for the generalised Lie derivative is identical to that in \cite{advances} evaluated at $\omega=\sk_x$.
\end{remark}

The other approach is to fix the smoothing kernel $\sk \in \SK(M)$ so that $ x \mapsto F(\sk)(x)$ is a smooth function of $x$. We may then define another Lie derivative of $F$ (which we denote $\Liet_X F$) by fixing $\sk$ and taking the (ordinary) Lie derivative of $F(\sk)$, so that
\begin{equation}
(\Liet_X F)(\sk)\coleq \Lie_X(F(\sk)).\label{lie1}
\end{equation}

Having defined suitable derivatives on $\MEF(M)$ and established that $\MC(M)$ is non-void, we turn to the definition of moderate and negligible functions on manifolds. We start with the definition of negligible functions. Consider a net $\SO_\e$ of smoothing operators converging to the identity. Then from this point of view the natural definition of a negligible function $F$ is one that satisfies $F(\SO_\e) \to 0$ as $k \to \infty$ in $C^\infty(M)$ (i.e. in all derivatives).  Writing this in terms of smoothing kernels we therefore require $\Lie_{X_1}\ldots\Lie_{X_k} (F(\sk_{\eps})) \to 0$ as $\eps \to 0$. Since $(\Liet_X F)(\sk)=\Lie_X(F(\sk))$ this automatically gives stability of the subspace of negligible functions under the ordinary Lie derivative $\Liet_X$. However we also require stability of negligible functions under the generalised Lie derivative $\gLie_X$. This suggests that we require 
\begin{equation*}
(\Liet_{\tilde X_1}\ldots\Liet_{\tilde X_k}  \gLie_{X_1}\ldots\gLie_{X_\ell}  F)(\sk_{\eps}) \to 0 \qquad \hbox{as $\eps \to 0$}. 
\end{equation*}
However, by definition we have
\begin{equation*}
((\Liet_X-\gLie_X)F)(\sk_\eps)=dF(\sk_\eps)(\Liesk_X \sk_\eps) 
\end{equation*}
so that taking linear combinations of the two types of Lie derivative is equivalent to looking at $dF$ and evaluating it on the tangent space to $\MC(M)$. We therefore introduce the space 
\begin{equation*}
\MC_0(M)\coleq\{\skk_0 \in C^\infty(M, \Omega^n_c(M))^{(0,1]}: \sk \in \MC(M) \Rightarrow \skk_0+\sk \in \MC(M)\}
\end{equation*}
and make the following definition:
\begin{definition}[Negligible functions] \label{nullman}
The function $F \in \MEF(M)$ is negligible if for any given compact $K\subset M$ $\forall k, j, m \in\Nat_0$ $\forall X_1\ldots X_k \in \TM(M)$ $\forall \sk \in \MC(M)$ $\forall \skk_1 \ldots \skk_j \in \MC_0(M)$:
\begin{equation*} 
\sup_{x\in K} \abso{ (\Liet_{X_1}\ldots\Liet_{X_k} (d^jF)(\sk_\e)(\skk_{1,\e} \ldots\skk_{j,\e})(x) } = O(\eps^{m}) \qquad \hbox{as } \eps \to 0. 
\end{equation*}
The set of negligible elements is denoted $\CN(M)$.
\end{definition}
In order that the space of negligible functions is an ideal we also need to restrict to the space of moderate functions.

\begin{definition}[Moderate functions] \label{moderateman}
The function $F \in \MEF(M)$ is moderate if for any given compact $K\subset M$ $\forall k, j \in\Nat_0$ $\forall \sk \in \MC(M)$ $\forall \skk_1 \ldots \skk_j \in \MC_0(M)$ $\exists N \in \Nat_0$ $\forall X_1\ldots X_k \in \TM(M)$:
\begin{equation} 
\sup_{x\in K} \abso{ (\Liet_{X_1}\ldots\Liet_{X_k} (d^jF)(\sk_\e)(\skk_{1,\e} \ldots\skk_{j,\e})(x) }
  = O(\eps^{-N}) \qquad \hbox{as } \eps \to 0. \label{32b}
\end{equation}
The set of moderate elements of $\MEF(M)$ is denoted $\CM(M)$.
\end{definition}

\begin{remark}
Although the above definitions require one to consider derivatives $d^jF$ of arbitrary order in practice one only needs to verify this condition is satisfied by objects that are embedded into the algebra via $\sigma$ or $\iota$. Since $\sigma$ does not depend on $\sk$ and the embedding $\iota$ is linear in $\sk$, this leaves the cases $j=0$ and $j=1$.
\end{remark}

\begin{theorem}\label{ideal}
\leavevmode
\begin{enumerate}[label=(\alph*)]
\item $\CM(M)$ is  a subalgebra of $\MEF(M)$.
\item $\CN(M)$ is an ideal in $\CM(M)$.
\end{enumerate}
\end{theorem}

\begin{proof}
Because of the property of derivatives it is clear from the definitions that that the product of two moderate functions is moderate and the product of a negligible function with a moderate function is negligible.
\end{proof}

\begin{proposition} \label{nulltest}
Let $F \in \CM(M)$ be such that for all compact $K\subset M$ $\forall m \in\Nat_0$ $\forall \sk \in \MC(M)$ 
\begin{equation*} 
\sup_{x\in K} \abso{ F(\sk_{\eps})(x) }  = O(\eps^{m}) \qquad \hbox{as } \eps \to 0. 
\end{equation*}
Then $F \in \CN(M)$.
\end{proposition}

\begin{proof} 
This follows from looking at $F(\sk_\eps+\eps^k\skk_\eps)$ where $\sk \in \MC(M)$, $\skk \in \MC_0(M)$, applying the mean-value theorem and using the definition of moderateness of $F$ with $k$ suitably chosen.
\end{proof}

This result shows that one does not need derivatives to test negligibility of a moderate function.

\begin{theorem} 
Let $X\in \mathfrak{X}(M)$. Then
\begin{enumerate}[label=(\alph*)]
\item $\gLie_X \CM(M) \subseteq \CM(M)$ and $\Liet_X \CM(M) \subseteq \CM(M)$.
\item $\gLie_X \CN(M) \subseteq \CN(M)$ and $\Liet_X \CN(M) \subseteq \CN(M)$.
\end{enumerate}
\end{theorem}
This follows immediately from the definitions.

We are finally in a position to define generalised functions on manifolds.
\begin{definition}[Generalised Functions]
The space
\begin{equation*}
\CGM(M) = {\CM(M) \over \CN(M) } 
\end{equation*}
is called the Colombeau algebra of generalised functions on $M$.
\end{definition}
 
\begin{theorem}\label{Col}
The space of Colombeau generalised functions $\CGM(M)$ is a fine sheaf of associative commutative differential algebras on $M$.
\end{theorem}

\begin{proof}
By construction the basic space $\MEF(M)$ is an associative commutative differential algebra, with derivative the generalised Lie derivative $\gLie$ given by equation \eqref{liealg}. $\CM(M)$ is a subalgebra of $\MEF(M)$ and $\CN(M)$ is an ideal in $\CM(M)$, hence $\CGM(M)$ is an algebra. Furthermore the spaces $\CM(M)$ and $\CN(M)$ are stable under both the generalised and ordinary Lie derivatives so that $\CGM(M)$ is a differential algebra with respect to the both Lie derivatives. The sheaf properties of $\CGM(M)$ follow from the localisation results \cite{specfull}.
\end{proof}

We now want to show that we may embed the space of distributions $\cD'(M)$ in the space of generalised functions $\CGM(M)$. Given a distribution $T$ in $\cD'(M)$ we define the function $\tilde T \in \MEF(M)$ by
\begin{equation*}
\tilde T(\sk)(x)=\langle T, \sk_x \rangle.
\end{equation*}
We now need to show that $\tilde T$ is moderate. For this we need to look at Lie derivatives of $d^j\tilde T(\sk)$. For $j=0$ we have $\tilde T(\sk_\eps)=\langle T, \sk_\eps \rangle$ and it then follows from property (3) of the smoothing kernels that
\[ \Lie_{X_1}\ldots\Lie_{X_k}(\tilde T(\sk_\eps))(x)  =\langle T, \Lief_{X_1}\ldots\Lief_{X_k}\sk_{x,\eps} \rangle = O(\eps^{-N}). \]
Since the embedding is linear, we have for $j=1$ that $(d\tilde T)(\sk)(\Psi)=\langle T, \Psi \rangle$, so the above argument gives the desired bound on the growth, while for $j \geq 2$ we have $d^j\tilde T=0$. This shows that $\tilde T$ is moderate and we have an embedding
\begin{align*}
\iota \colon \cD'(M) &\to \CGM(M) \\
T &\mapsto [\tilde T] 
\end{align*}
where $[\tilde T]$ is the equivalence class of $\tilde T$ in
$\CGM(M)$. 

By the definition of the generalised Lie derivative we have
\begin{align*}
\gLie_X(\iota T)(\sk)(x)
&=-\langle T, (\Liesk_X \sk)_x \rangle + \Lie_X \langle T, \sk_x \rangle \\
&=-\langle T, \Lien_X\sk_x \rangle \\
&=\langle \Lie_X T, \sk_x \rangle \\
&=\iota(\Lie_X T)(\sk)(x).
\end{align*}
Hence,
\begin{equation}
\gLie_X(\iota T)=\iota(\Lie_X T)
\end{equation}
and thus the embedding $\iota$ commutes with the generalised Lie derivative.

It is clear that if  $f$ is a smooth function on $M$ then $\hat f$ defined by $\hat f(\sk)(x)=f(x)$ is a moderate function. By passing to the equivalence class $[\hat f]$ we obtain the embedding $\sigma \colon C^\infty(M) \to \CGM(M)$ from above. Clearly $\sigma$ gives an injective algebra homomorphism of the algebra of smooth functions on $M$ into $\CGM(M)$, the algebra of generalised functions on $M$. Furthermore since $\sigma(f)$ has no dependence on $\sk$ we only have the second term in the formula for the definition of the generalised Lie derivative so $\sigma$ also commutes with the Lie derivative. Finally, it easily follows from Definition \ref{kernels} that for a smooth function the difference between $\tilde f$ and $\hat f$ is negligible and hence on passing to the quotient $\iota$ coincides with $\sigma$ on $C^\infty(M)$.

Collecting these results together we have the following theorem:
\begin{theorem}
$\iota \colon \cD'(M) \to \CGM(M)$ is a linear embedding that commutes with the generalised Lie derivative and coincides with $\sigma \colon C^\infty(M)\to \CGM(M)$ on $C^\infty(M)$. Thus $\iota$ renders $\cD'(M)$ a linear subspace and $C^\infty(M)$ a faithful subalgebra of $\CGM(M)$.
\end{theorem}

As explained in the introduction the concept of association is an important feature of the theory of Colombeau algebras on manifolds as in many cases it allows us to recover a description in terms of classical distributions by a method of `coarse graining'. We now show how this notion may be extended to generalised functions on manifolds.
\begin{definition}[Association]
We say an element $[F]$ of $\CGM(M)$ is associated to $0$ (denoted $[F] \approx 0$) if for each $\omega \in \Omega^n_c(M)$ we have
\begin{equation*}
\lim_{\eps \to 0}\int_{x \in M}F(\sk_{\eps})(x)\omega(x)=0 \quad \forall \sk \in \MC(M). 
\end{equation*}
We say two elements $[F],[G]$ are associated and write $[F]\approx [G]$ if $[F-G] \approx 0$.
\end{definition}

\begin{definition}[Associated distribution]
We say $[F] \in \CGM(M)$ admits $u \in \cD'(M)$ as an associated distribution if for each $\omega \in \Omega^n_c(M)$  we have
\begin{equation*}
\lim_{\eps \to 0}\int_{x \in M}F(\sk_{\eps})(x)\omega(x)=\langle u, \omega \rangle \quad \forall \sk \in \MC(M). 
\end{equation*}
\end{definition}
Again, these definitions do not depend on the representative of the class. As in $\Real^n$ at the level of association we regain the usual results for multiplication of distributions, provided that suitable $L^1$-conditions like \eqref{l1cond} and \eqref{l1cond2} are used.

\begin{proposition} \label{assocprop}
\leavevmode
\begin{enumerate}[label=(\alph*)]
 \item If $f \in C^{\infty}(M)$ and $T \in \cD'(M)$ then 
\begin{equation*}
\iota(f)\iota(T) \approx \iota(fT). 
\end{equation*}
\item If $f, g \in C^{0}(M)$ then 
\begin{equation*}
\iota(f)\iota(g) \approx \iota(fg).  
\end{equation*}
\end{enumerate}
\end{proposition}

The above results establish almost everything we want at the scalar level. Before going on to look at the tensor theory and develop a theory of differential geometry there is one further ingredient we will require, which is the notion of directional (or covariant) derivative $\nabla_XF$ of a generalised scalar field. Ideally this would be $C^\infty(M)$-linear in $X$ (so that $\nabla_{fX}F=f\nabla_XF$) and commute with the embedding. However, it is not hard to see that this is not possible since this would require that
\begin{align*}
\iota(f\nabla_Xg) &= \iota(\nabla_{fX}g) \cr
&= \nabla_{fX}\iota(g) \cr
&= f\nabla_X(\iota g) \cr
&=\iota(f)\iota(\nabla_X g)
\end{align*}
which cannot in general be true by the Schwartz impossibility result. However, in view of Proposition \ref{assocprop} a $C^\infty(M)$-linear derivative that commutes with the embedding only at the level of association is not ruled out by the Schwartz result.  

By thinking of $F(\sk)(x)$ for fixed $\sk$ as a function of $x$ we may make the following definition of generalised covariant derivative
\begin{definition}[Covariant derivative of a generalised scalar field]
Let $F \in \CGM(M)$ be a generalised scalar field and $X$ a smooth vector field. Then we define the covariant derivative $\tilde \nabla_XF$ by
\[ (\tilde \nabla_XF)(\sk)=\nabla_X(F(\sk)). \]
\end{definition}
We note that almost by definition this satisfies the requirements of a covariant derivative and for the case of a scalar field (which we are considering here) this is identical to the Lie derivative $\Liet_XF$ given by \eqref{lie1} and hence is well defined. Although it is $C^\infty(M)$-linear in $X$ this derivative does not commute with the embedding into $\CGM(M)$. However as we now show this derivative {\it does commute} with the embedding at the level of {\it association}.

\begin{proposition}
Let $T \in \cD'(M)$ and $X$ be a smooth vector field; then
\begin{equation}
\tilde \nabla_X\iota(T)=\Liet_X\iota(T) \approx \iota(\Lie_XT)=\iota(\nabla_XT)
\end{equation}
\end{proposition}

\begin{proof}
In the following calculation let $\sk_{\eps}$ be a fixed delta net of smoothing kernels. Given $\mu$ a smooth $n$-form of compact support then
\begin{align*}
 \lim_{\e \to 0} \int_M ( \tilde \Lie_X \iota(T))(\sk_\e) \mu & = \lim_{\e \to 0} \int_M ( \Lie_X (\iota T(\sk_\e))) \mu \\
&= - \lim_{\eps \to 0}\int_M (\iota T)(\sk_\e)(\Lie_X \mu) \\
&= - \lim_{\eps \to 0}\int_M \langle T, \sk_{x,\eps} \rangle (\Lie_X \mu)(x) \\
&= \langle T, -\Lie_X\mu \rangle \\\
&= \langle \Lie_XT, \mu \rangle. \qedhere
\end{align*}
\end{proof}

\section{Conclusion}

In this paper we have reviewed the construction of the Colombeau algebra on $\Real^n$ and adapted it to define the Colombeau algebra on a manifold $M$. The key idea has been to look at the construction on manifolds first of all in terms of smoothing operators and then translate this into the language of smoothing kernels. The result of this is to replace the mollifiers $\mol(\y-\x)$ by smoothing kernels $\sk_x(y)$ and the scaled mollifiers $\mol_\eps(\y-\x)$ by delta nets of smoothing kernels $\sk_{x,\eps}(y)$.  In this way, given a locally integrable function $f$ we may approximate it by a 1-parameter family of smooth functions (depending on $\sk$) according to
\[ \tilde f_\eps(x)=\int_{y \in M}f(y)\sk_{x,\eps}(y). \]
For fixed $\sk \in \MC(M)$ these may be treated just like smooth functions on manifolds so all the standard operations that may be carried out on smooth functions extend to the smoothed functions $\tilde f_\eps$.  The embedding extends to distributions $T \in \cD'(M)$ by defining $\tilde T_\eps(x)=\langle T, \sk_{x,\eps}\rangle$. The nets of smoothing kernels tend to $\delta_x$ as $\eps \to 0$ and by using the rate at which this happens to introduce a grading $\MC(M)$ on the smoothing kernels we have a condition which corresponds to the vanishing moment condition on $\Real^n$.  We can therefore define the spaces of moderate and negligible functions which allows us to define $\CGM(M)$ as the quotient $\CGM(M)=\CM(M)/\CN(M)$.  The algebra of generalised functions $\CGM(M)$ contains the space of smooth functions as a subalgebra and has the space of distributions as a canonically embedded linear subspace. We also introduced the generalised Lie derivative which commutes with the embedding and makes $\CGM(M)$ into a differential algebra. Finally we defined the covariant derivative of generalised scalar fields on the manifold $M$ and showed that this commutes with the distributional (covariant) derivative at the level of association. In a subsequent paper \cite{diffgeom} this theory will be extended to a nonlinear theory of tensor distributions on a manifold $M$ where this is used to develop a theory of nonlinear distributional geometry.

{\bfseries Acknowledgments.} E.~Nigsch was supported by grants P26859 and P30233 of the Austrian Science Fund (FWF).


\begin{thebibliography}{99}

\bibitem{diffgeom}
Nigsch EA, Vickers JA. 2019  A nonlinear theory of distributional geometry.
  Preprint.

\bibitem{GT}
Geroch R, Traschen J. 1987  Strings and other distributional sources in general
  relativity. {\em Phys. Rev. D (3)} \textbf{36}, 1017--1031.

\bibitem{col1}
Colombeau JF. 1984 {\em New generalized functions and multiplication of
  distributions}.
Number~84 in North-Holland Mathematics Studies. Amsterdam: North-Holland
  Publishing Co.

\bibitem{cvw}
Clarke C, Vickers J, Wilson J. 1996  Generalized functions and distributional
  curvature of cosmic strings. {\em Classical Quantum Gravity} \textbf{13},
  2485--2498.

\bibitem{balasin}
Balasin H. 1997  Distributional energy-momentum tensor of the extended Kerr
  geometry. {\em Class. Quantum Grav.} \textbf{14}, 3353--3362.

\bibitem{steinbauer}
Steinbauer R. 1997  The ultrarelativistic Riessner-Nordstr\o m field in the
  Colombeau algebra. {\em J. Math. Phys.} \textbf{38}, 1614--1622.

\bibitem{SV}
Steinbauer R, Vickers JA. 2006  The use of generalized functions and
  distributions in general relativity. {\em Classical Quantum Gravity}
  \textbf{23}, r91--r114.

\bibitem{col2}
Colombeau JF. 1992 {\em Multiplication of distributions}.
Number 1532 in Lecture Notes in Mathematics. Berlin: Springer-Verlag.

\bibitem{impossibility}
Schwartz L. 1954  Sur l'impossibilit{\'e} de la multiplication des
  distributions. {\em Comptes Rendus de l'Acad{\'e}mie des Sciences}
  \textbf{239}, 847--848.

\bibitem{VW1}
Vickers J, Wilson J. 1999  Invariance of the distributional curvature of the
  cone under smooth diffeomorphisms. {\em Classical Quantum Gravity}
  \textbf{16}, 579--588.

\bibitem{colmeril}
Colombeau JF, Meril A. 1994  Generalized functions and multiplication of
  distributions on {${\mathcal C}\sp \infty$} manifolds. {\em J. Math. Anal.
  Appl.} \textbf{186}, 357--364.

\bibitem{advances}
Grosser M, Kunzinger M, Steinbauer R, Vickers JA. 2002  A Global Theory of
  Algebras of Generalized Functions. {\em Adv. Math.} \textbf{166}, 50--72.

\bibitem{KM}
Kriegl A, Michor PW. 1997 {\em The convenient setting of global analysis}.
Number~53 in {Mathematical Surveys and Monographs}. Providence, RI: American
  Mathematical Society.

\bibitem{MO}
Oberguggenberger M. 1992 {\em Multiplication of Distributions and Applications
  to Partial Differential Equations}.
Number 259 in Pitman Research Notes in Mathematics. Harlow, U.K.: Longman.

\bibitem{book}
Grosser M, Kunzinger M, Oberguggenberger M, Steinbauer R. 2001 {\em Geometric
  theory of generalized functions with applications to general relativity}.
Number 537 in Mathematics and its Applications. Dordrecht: Kluwer Academic
  Publishers.

\bibitem{jelinek}
Jel{\'i}nek J. 1999  {An intrinsic definition of the {C}olombeau generalized
  functions.}. {\em Commentat. Math. Univ. Carol.} \textbf{40}, 71--95.

\bibitem{mem}
{Grosser} M, {Farkas} E, {Kunzinger} M, {Steinbauer} R. 2001  On the
  foundations of nonlinear generalized functions {I, II}. {\em Mem. Am. Math.
  Soc.} \textbf{729}.

\bibitem{VW2}
Vickers JA, Wilson JP. 1998  {A nonlinear theory of tensor distributions}. {\em
  ESI-Preprint (available electronically at
  http://www.esi.ac.at/ESI-Preprints.html)} \textbf{566}.

\bibitem{Nigsch}
Nigsch EA. 2015  The functional analytic foundation of {Colombeau} algebras.
  {\em J. Math. Anal. Appl.} \textbf{421}, 415--435.

\bibitem{distcurv}
Nigsch EA. 2019  Spacetimes with distributional semi-Riemannian metrics and
  their curvature. Submitted.

\bibitem{specfull}
Grosser M, Nigsch EA. 2018  Full and special {C}olombeau Algebras. {\em Proc.
  Edinb. Math. Soc} \textbf{61}, 961--994.

\end{thebibliography}
\end{document}